\documentclass[12pt]{amsart}
\usepackage{amscd,verbatim}
\usepackage{mathrsfs,yfonts,times}
\usepackage{amssymb,array}
\usepackage[all]{xy}
\usepackage{url}

\newcommand{\G}{\mathbb{G}}

\renewcommand{\P}{\mathbb{P}}
\newcommand{\Q}{\mathbb{Q}}
\newcommand{\Z}{\mathbb{Z}}

\newcommand{\Hom}{\operatorname{Hom}}

\newcommand{\Coker}{\operatorname{Coker}}
\newcommand{\coker}{\operatorname{coker}}

\newcommand{\Spec}{\operatorname{Spec}}

\newcommand{\Jac}{\operatorname{Jac}}

\newcommand{\Tor}{{\operatorname{Tor}}}
\newcommand{\Div}{\operatorname{Div}}
\newcommand{\ord}{\operatorname{ord}}
\newcommand{\Res}{\operatorname{Res}}
\newcommand{\Gal}{\operatorname{Gal}}

\newcommand{\SL}{\operatorname{SL}}

\newcommand{\ab}{{\operatorname{ab}}}
\newcommand{\geo}{{\operatorname{geo}}}

\newcommand{\ur}{{\operatorname{ur}}}
\newcommand{\uab}{{\operatorname{uab}}}
\newcommand{\cyc}{{\operatorname{cyc}}}

\renewcommand{\lim}{\operatornamewithlimits{\varprojlim}}
\newcommand{\colim}{\operatornamewithlimits{\varinjlim}}

\renewcommand{\deg}{\operatorname{deg}}
\newcommand{\ol}{\overline}

\renewcommand{\epsilon}{\varepsilon}
\renewcommand{\div}{{\operatorname{div}}}

\newcommand{\tJ}{{\widetilde{J}}}

\newcommand{\ul}{\underline}

\numberwithin{equation}{section}

\def\SL{\mathrm{SL}}

\def\M#1#2#3#4{\begin{pmatrix}#1&#2\\#3&#4\end{pmatrix}}
\def\SM#1#2#3#4{\left(\begin{smallmatrix}#1&#2\\#3&#4\end{smallmatrix}\right)}

\def\JS#1#2{\left(\frac{#1}{#2}\right)}

\newtheorem{lemma}{Lemma}[subsection]

\newtheorem{theorem}[lemma]{Theorem}

\newtheorem{proposition}[lemma]{Proposition}
\newtheorem{Proposition}[lemma]{Proposition}
\newtheorem{Lemma}[lemma]{Lemma}

\newtheorem{definition}[lemma]{Definition}

\theoremstyle{remark}

\newtheorem{remark}[lemma]{Remark}
\newtheorem{Remark}[lemma]{Remark}

\newtheorem{example}[lemma]{Example}

\numberwithin{equation}{subsection}

\begin{document}
\title[Maximal abelian extension of $X_0(p)$]
{Maximal abelian extension of $X_0(p)$ 
unramified outside cusps}

\author{Takao Yamazaki \and Yifan Yang}
\date{\today}
\address{Mathematical Institute, Tohoku University,
Aoba, Sendai 980-8578, Japan}
\email{ytakao@math.tohoku.ac.jp}
\address{Department of Mathematics, 
National Taiwan University and National Center for Theoretical
Sciences, Taipei, Taiwan 10617}
\email{yangyifan@ntu.edu.tw}

\begin{abstract}
Let $p$ be a prime number.
Mazur proved that a geometrically maximal unramified abelian covering of $X_0(p)$
over $\Q$ 
is given by the Shimura covering $X_2(p) \to X_0(p)$,
that is, a unique subcovering of $X_1(p) \to X_0(p)$
of degree $N_p := (p-1)/\gcd(p-1, 12)$.
In this short paper,
we show that a geometrically maximal abelian covering 
$X_2'(p) \to X_0(p)$ of $X_0(p)$ over $\Q$
unramified outside cusps is cyclic of degree $2N_p$.
The main ingredient for the construction of $X_2'(p)$ 
is the generalized Dedelind eta functions.
\end{abstract}

\maketitle

\section{Introduction}
\subsection{Main result}
Let $p$ be a prime.
We consider the modular curves 
$X_0(p)$ and $X_1(p)$ 
as geometrically integral smooth proper curves over $\Q$.
We choose a model of $X_1(p)$ over $\Q$
such that the cusp at infinity splits completely
in a finite cyclic covering $f_p : X_1(p) \to X_0(p)$ 
of degree $(p-1)/2$.
Note that $f_p$ is possibly ramified.
We denote by $\ul{f}_p : X_2(p) \to X_0(p)$
its (unique) subcovering of degree $N_p :=(p-1)/\gcd(p-1, 12)$,
which is the maximal unramified subcovering of $f_p$,
called the {\it Shimura covering}.
We recall an important result due to Mazur.
\begin{theorem}[{\cite[Theorem 2]{Mazur}}]\label{thm:mazur}
The Shimura covering
$\ul{f}_p : X_2(p) \to X_0(p)$ is geometrically maximal
over $\Q$.
\end{theorem}
Here, we say a finite unramified abelian covering
$Y' \to Y$ of smooth geometrically integral curves over $\Q$
is {\it geometrically maximal} over $\Q$
if any finite unramified abelian covering of $Y$
is isomorphic to a subcovering of
$Y' \times_{\Spec \Q} \Spec \Q^\ab \to Y$,
where $\Q^\ab$ is a maximal abelian extension of $\Q$
(see Definition \ref{def:geo-max-cov}).

Let $Y_0(p) \subset X_0(p)$
be an open subscheme
such that $X_0(p) \setminus Y_0(p)$
consists of all (two) cusps.
In this short note,
we shall construct a geometrically maximal 
unramified covering of $Y_0(p)$ and prove 
the following result.

\begin{theorem}\label{thm:main1}
There exists a cyclic covering
$\ul{f}_p' : X_2'(p) \to X_0(p)$ of degree $2N_p$
such that
\begin{enumerate}
\item 
$\ul{f}_p'|_{Y_2'(p)} : Y_2'(p) \to Y_0(p)$ is
a geometrically maximal unramified covering of $Y_0(p)$ over $\Q$,
where $Y_2'(p):={\ul{f}_p'}^{-1}(Y_0(p))$;
and 
\item
$\ul{f}_p'$ factors as
$\ul{f}_p'=\ul{f}_p \circ \pi_p$,
where $\pi_p : X_2'(p) \to X_2(p)$
is a degree two covering
that is unramified outside cusps.
\end{enumerate}
\[
\xymatrix{
& X_2'(p) \ar[d]_{\pi_p} \ar[dr]^{\ul{f}_p'}&
\\
X_1(p) \ar[r]^{} \ar@/_2em/[rr]_{f_p} &
X_2(p) \ar[r]^{\ul{f}_p} &
X_0(p).
}
\]
\end{theorem}

We shall prepare some general facts 
on abelian coverings of smooth curves
in \S \ref{sec:covering}.
We then prove Theorem \ref{thm:main1} in \S \ref{sec:modular}.
The function field $\Q(X_2'(p))$ 
of $X_2'(p)$ will be obtained as
a quadratic extension of $\Q(X_2(p))$
generated by a square root of 
an explicitly constructed rational function in
$\Q(X_2(p))$.
A key ingredient for this construction 
is the {\it generalized Dedekind eta functions},
which we recall in \S \ref{sec:ded-eta}.

\subsection{Notations and conventions}\label{sect:notation}
Let $G$ be a profinite group.
We write $G^\ab$ for the 
quotient of $G$ by the closure of its commutator subgroup.
For a $G$-module $M$,
we denote by 
$M^{G} \subset M$ and $M \twoheadrightarrow M_{G}$ 
its $G$-invariant subgroup and $G$-coinvariant quotient, respectively.
When $M$ is locally compact,
we write $M^*$ for the Pontryagin dual of $M$.
In practice, we shall only consider the cases
where $M$ is either a discrete torsion group
or a compact free $\hat{\Z}$-module of finite rank,
hence $M$ is a $\hat{\Z}$-module 
and $M^* = \Hom(M, \Q/\Z)$,
where $\Hom$ means the group of $\hat{\Z}$-linear maps.
Note also that we have canonical isomorphisms
\begin{equation}\label{eq:dual-inv-coinv}
(M^*)^G \cong (M_G)^*
\quad \text{and} \quad
(M^G)^* \cong (M^*)_G.
\end{equation}
Indeed, both sides of the first isomorphism are identified with
the group of continuous $G$-equivariant homomorphisms $M \to \Q/\Z$.
The second follows from the first
by replacing $M$ by $M^*$ and $M^{**} \cong M$.

Let $S$ be a scheme which is 
separated and of finite type over a field.
We write $\G_{m, S}$ for the multiplicative group scheme over $S$,
and $\mu_{n, S} :=\ker(n : \G_{m, S} \to \G_{m, S})$.
When $S=\Spec R$ is affine, we write $G_{\Spec R}=G_R$.
For an \'etale sheaf $F$ on $S$,
we write $H^*(S, F)$ (resp. $H^*_c(S, F)$)
for the \'etale cohomology 
(resp. the \'etale cohomology with compact support).
A $G_k$-module $M$ is identified with an \'etale sheaf on $\Spec k$
and we write $H^i(\Spec k, M)=H^i(k, M)$.
Suppose now $S$ is connected.
We write $\pi_1(S)^\ab := \pi_1(S, x)^\ab$,
where $\pi_1(S, x)$ is the \'etale fundamental group of $S$
with respect to some geometric point $x \to S$
(on which $\pi_1(S, x)^\ab$ does not depend up to unique isomorphism).
In particular, we have a canonical isomorphism
\begin{equation}\label{eq:pi-H1-duality}
\pi_1(S)^\ab \cong H^1(S, \Q/\Z)^*  
\end{equation}

Let $k$ be a field of characteristic zero.
We take an algebraic closure $\ol{k}$ and
put $G_k:=\Gal(\ol{k}/k)$.
We write 
$\Z/n\Z(r)=\mu_{n, k}(\ol{k})^{\otimes r}$ if $r \ge 0$,
and 
$\Z/n\Z(r)=\Hom(\Z/n\Z(-r), \Z/n\Z)$ if $r < 0$.
Let $M$ be a $G_k$-module.
We write $M[n]:=\{ x \in M \mid nx=0 \}$.
We define the (twisted) Tate module by 
$TM(r) := \lim_n M[n] \otimes \Z/n\Z(r)$.
If $M$ is torsion,
then we define $M(r):=M \otimes_{\hat{\Z}} \hat{\Z}(r)$,
where $\hat{\Z}(r):=\lim_n \Z/n\Z(r)$.
For general $M$,
we define its maximal $\mu$-type subgroup by
\begin{equation}\label{eq:def-mu}
M^\mu := 
\{ a \in M_\Tor
\mid \sigma(a)=\chi_\cyc(\sigma) a
~\text{for all}~ \sigma \in G_k \},
\end{equation}
where $\chi_\cyc : G_k \to \hat{\Z}^\times$
is the cyclotomic character.
(In other words, 
$M^\mu(-1)=M_\Tor(-1)^{G_k}$.)

A commutative algebraic group $A$ over $k$
defines a $G_k$-module $A(\ol{k})$,
and hence an \'etale sheaf on $\Spec k$.
In this case, to ease the notation
we often write $A$ for $A(\ol{k})$,
e.g. 
$A[n]=A(\ol{k})[n], TA(r)=TA(\ol{k})(r)$ 
and $A^\mu = A(\ol{k})^\mu$.

\section{Abelian coverings of smooth curves}\label{sec:covering}
In this section, we collect basic 
facts about abelian fundamental group
of a smooth curve.

\subsection{Setting}\label{sec:setting}
Let $X$ be a smooth proper geometrically integral curve over 
a field $k$ of characteristic zero.
Let $D$ be an effective reduced divisor on $X$
and set $Y:= X \setminus D$.
We write $g_X$ for the genus of $X$, and set
\begin{equation}\label{eq:def-rD}
r_D:=\sum_{x \in D} [k(x):k] - 1.
\end{equation}
We suppose that $X$ admits a degree one divisor.
Note that
this implies the existence of
a degree one divisor $E$ supported on $Y$
by the approximation lemma
(see e.g. \cite[{p. 12}]{serreLF}).
Indeed, 
let $E_1=\sum n_x x$ be a degree one divisor on $X$.
For each $x \in |D| \cap |E_1|$ 
(resp. $x \in |D| \setminus |E_1|$),
let $a_x \in k(X)$ be such that $\ord_x(a_x)=n_x$ 
(resp. $a_x=1$).
By applying the cited result 
to the family $(a_x)_{x \in |D|}$,
we obtain $f \in k(X)$ such that
$\ord_x(f)=n_x$ for all $x \in |D| \cap |E_1|$
and $\ord_x(f)=0$ for all $x \in |D| \setminus |E_1|$.
Then $E:=E_1-\div(f)$ is a degree one divisor
supported on $Y$.

\subsection{Geometrically maximal unramified abelian covering}
Let $k^\ab$ be the maximal abelian extension of $k$.
We have $G_k^\ab = \Gal(k^\ab/k)$ and
$\pi_1(Y)^\ab=\Gal(k(Y)^{\ur, \ab}/k(Y))$,
where $k(Y)^{\ur, \ab}$ is the maximal unramified abelian extension
of the function field $k(Y)$ of $Y$.
Denote by $\pi_1(Y)^{\ab, \geo}$
the kernel of the canonical map
$\pi_1(Y)^\ab \to G_k^\ab$.
We have an exact sequence
\begin{equation}\label{eq:ex-seq1}
0 \to \pi_1(Y)^{\ab, \geo} \to \pi_1(Y)^\ab \to G_k^\ab \to 0.
\end{equation}

\begin{remark}\label{rem:KL}
It is shown in \cite[Theorem 1]{KL} that
$\pi_1(Y)^{\ab, \geo}$ is finite
if $k$ is finitely generated over its prime subfield.
\end{remark}

\begin{definition}\label{def:geo-max-cov}
We say a finite unramified abelian covering
$Y' \to Y$
is {\it geometrically maximal} over $k$
if 
$Y'$ is geometrically integral over $k$
and if the composition 
$\pi_1(Y)^{\ab, \geo} \hookrightarrow \pi_1(Y)^{\ab} \twoheadrightarrow \Gal(Y'/Y)$
is an isomorphism:
\[
\xymatrix{
& k(Y)^{\ur, \ab} & 
\\
k(Y') \ar[ur] & & k^\ab(Y) \ar[ul]_{\pi_1(Y)^{\ab, \geo}}
\\
& k(Y). \ar[ul]^{\Gal(Y'/Y)} \ar[ur]_{G_k^\ab} &
}
\]
(Equivalently, 
any finite unramified abelian covering $Y'' \to Y$
is a subcovering of $Y' \times_{\Spec k} \Spec k^\ab \to Y$.
See also Remark \ref{def:geomax} 
and Proposition \ref{prop:ex-gm} below.)
\end{definition}

When $Y$ has a $k$-rational point $x$,
a maximal abelian unramified covering of $Y$
in which $x$ splits completely
yields a geometrically maximal covering.
There is, however, no such a description 
if no $k$-rational point is available.
Given $Y$, we are interested 
in finding a geometrically maximal unramified abelian covering.

\subsection{Generalized Jacobian}
Let $\tJ := \Jac(X, D)$ be 
the generalized Jacobian of $X$ with modulus $D$
in the sense of Rosenlicht-Serre \cite{Serre},
which is a semi-abelian variety over $k$.
It fits in an exact sequence
\begin{equation}\label{eq:tj-ext}
 0 \to \G_D \to \tJ \to J \to 0,
\end{equation}
where $J=\Jac(X)$ is the Jacobian variety of $X$,
and
\begin{equation}\label{eq:def-GD} 
\G_D := \Coker[\G_{m, k} 
\to \bigoplus_{x \in D} \Res_{k(x)/k} \G_{m, k(x)}].
\end{equation}
We also recall that, 
if we put $\G_{X, D} := \ker(\G_{m, X} \to \G_{m, D})$,
there are isomorphisms
\begin{align}
\label{eq:tj-div}
\tJ(k) 
&\cong \Div^0(Y)/ 
\{ \div(f) \mid f \in k(Y)^\times, ~\ord_x(f-1) \ge 1 ~\text{for any}~ x \in D \}
\\
\label{eq:tj-h1}
&\cong \ker(H^1(X, \G_{X, D}) \to H^1(X, \G_{m, X}) 
\overset{\deg}{\to} \Z).
\end{align}
\subsection{A Galois module $M_\Tor$}
We define a divisible torsion $G_k$-module $M_\Tor$ by 
\begin{equation}\label{eq:def-mtor}
 M_\Tor := (T\tJ)^*(1). 
\end{equation}
(See \S \ref{sect:notation} for the notations for 
$T\tJ,~ (-)^*$ and $(-)(1)$).
Although this is not the torsion subgroup of a group $M$,
the notation $M_\Tor$ may be justified by the fact that
$M_\Tor$ agrees with the group of torsion points of 
the dual {\it $1$-motive} $M$ of $\tJ$
in the sense of \cite[Chapter V, \S 3]{Tate}.
(We will not use this fact.)
Let
\begin{equation}\label{eq:def-LGD} 
L_D := \Hom(\G_D, \G_{m, k})
\end{equation}
be the character group of $\G_D$ from \eqref{eq:def-GD}.
This is thus
a free abelian group of rank 
$r_D$ (see \eqref{eq:def-rD}),
equipped with a continuous $G_k$-action.
Recall that $g_X$ is the genus of $X$.

\begin{lemma}\label{prop:str-mtor}
There is an exact sequence of $G_k$-modules
\begin{equation}\label{eq:MforX}
0 \to J_\Tor \to M_\Tor \to L_D \otimes \Q/\Z \to 0.
\end{equation}
In particular, $M_\Tor \cong (\Q/\Z)^{\oplus(g_X+r_D)}$ 
as abelian groups.
\end{lemma}
\begin{proof}
The first statement follows from \eqref{eq:tj-ext}
and an isomorphism
$J_\Tor \cong (TJ)^*(1)$
deduced by
the Weil pairing $TJ \times J_\Tor \to \Q/\Z(1)$.
The last statement follows form \eqref{eq:MforX}.
\end{proof}

By taking the Tate twist $(-1)$ 
and the long exact sequence attached to \eqref{eq:MforX}
(see also \eqref{eq:def-mu}),
we get an exact sequence
that will be used later 
\begin{equation}\label{eq:exact-mtor}
 0 \to J_\Tor^\mu(-1)
\to M_\Tor^\mu(-1)
\to (L_D \otimes \Q/\Z)^\mu(-1)
\to H^1(k, J_\Tor(-1)).
\end{equation}

\begin{example}\label{ex:two}
Suppose that $D=P+Q$ consists of 
two $k$-rational points.
Then we have an exact sequence
\[ 0 \to J_\Tor \to M_\Tor \to \Q/\Z \to 0. \]
Explicitely, the $G_k$-module $M_\Tor = \colim_n M_\Tor[n]$ 
can be described as follows.
We have
\[ M_\Tor[n] = J[n] \oplus (\Z/n\Z) \]
as an abelian group,
and $G_k$-action is given by
\[ \sigma(\alpha, b) = (\sigma(\alpha) + b(\sigma(\beta)-\beta), b) \]
where $\sigma \in G_k, ~\alpha \in J[n], ~b \in \Z/n\Z$
and $\beta \in J(\ol{k})$ is a fixed element 
such that $n \beta$ is the class of $P-Q$ in $J(\ol{k})$
(see \eqref{eq:tj-div}).
\end{example}

\subsection{Abelian fundamental group}
The following result is essentially shown in \cite{KL},
but it is embedded in a (long) proof,
and hence we decide to include a proof to cut out the desired part.

\begin{proposition}\label{prop:exact-seq}
The exact sequence \eqref{eq:ex-seq1} splits (non-canonically),
and we have a canonical isomorphism
(see \S \ref{sect:notation} 
for the notations used in the right hand side)
\[ \pi_1(Y)^{\ab, \geo} 
\cong M_\Tor^\mu(-1)^*. 
\]
\end{proposition}

\begin{proof}
Set $\ol{Y}:=Y \times_{\Spec k} \Spec \ol{k}$.
The spectral sequence
$E_2^{i, j}=H^i(k, H^j(\ol{Y}, \Q/\Z)) \Rightarrow
H^{i+j}(Y, \Q/\Z)$
yields an exact sequence
\begin{align*}
0 
&\to H^1(k, \Q/\Z)
\to H^1(Y, \Q/\Z)
\to H^0(k, H^1(\ol{Y}, \Q/\Z))
\\
&\to H^2(k, \Q/\Z)
\overset{\phi}{\to} H^2(Y, \Q/\Z).
\end{align*}
Recall that by assumption 
we have a degree one divisor 
$E=\sum n_x x$ supported on $Y$ (see \S \ref{sec:setting}).
It gives rise to a section 
$\sum n_x N_{k(x)/k} \circ i_x^*$ of $\phi$,
where $i_x^* : H^2(Y, \Q/\Z) \to H^2(k(x), \Q/\Z)$ 
is the pull-back along 
the closed immersion $i_x : \Spec k(x)=x \to Y$
and $N_{k(x)/k} : H^2(k(x), \Q/\Z) \to H^2(k, \Q/\Z)$ 
is the norm (corestriction) map.
This shows the injectivity of  $\phi$.
By taking the dual sequence, we obtain a 
short exact sequence (see \eqref{eq:dual-inv-coinv})
\[ 0 \to (\pi_1(\ol{Y})^\ab)_{G_k}
\to \pi_1(Y)^\ab \to G_k^\ab \to 1,
\]
which splits again by the presence of $E$.
We have shown 
$(\pi_1(\ol{Y})^\ab)_{G_k} \cong \pi_1(Y)^{\ab, \geo}$.
We also have
$(T\tJ)_{G_k} \cong \Hom(M_\Tor^\mu(-1), \Q/\Z)$
by the definition \eqref{eq:def-mtor}.

Hence it remains to show an isomorphism of $G_k$-modules
\begin{equation}\label{eq:pi-tate-mtor}
 \pi_1(\ol{Y})^\ab \cong T\tJ 
\end{equation}
from which the proposition follows
by taking $G_k$-coinvariant quotients.
By the definitions of $\G_{X, D}$ from \eqref{eq:tj-h1},
there exists an exact sequence of \'etale sheaves
\[ 0 \to j_!(\mu_{n, Y}) \to \G_{X, D} \overset{n}{\to} \G_{X, D} \to 0 
\]
for any $n$,
where $j : Y \to X$ is the open immersion
so that $j_!(\mu_{n, Y}) \cong \ker(\mu_{n, X} \to \mu_{n, D})$.
It follows that
\begin{equation}\label{eq:hc-tj}
\tJ_\Tor[n] \cong H^1_c(\ol{Y}, \mu_n),\quad
\tJ_\Tor \cong H^1_c(\ol{Y}, \Q/\Z(1)),\quad
T\tJ \cong H^1_c(\ol{Y}, \hat{\Z}(1)).
\end{equation}
The last isomorphism
shows $T\tJ \cong H^1(\ol{Y}, \Q/\Z)^*$
by the Poincar\'e duality.
Now \eqref{eq:pi-tate-mtor} follows from
\eqref{eq:pi-H1-duality}.
\end{proof}

\begin{remark}
By applying the functor $A \mapsto A^*(1)$
to \eqref{eq:hc-tj},
we obtain isomorphisms
\begin{equation}\label{eq:hc-tj2}
M_\Tor[n] \cong H^1(\ol{Y}, \mu_n),\quad
M_\Tor \cong H^1(\ol{Y}, \Q/\Z(1)),\quad
TM_\Tor \cong H^1(\ol{Y}, \hat{\Z}(1)).
\end{equation}
\end{remark}

\subsection{Covering of $Y$}
Let $X'$ be another proper smooth geometrically integral curve over $k$
admitting a degree one divisor,
and let $f : X' \to X$ be a finite $k$-morphism.
Put $Y':=f^{-1}(Y)$ and $D':=X' \setminus Y'$.
(We regard $D'$ as a reduced effective divisor on $X'$.)
Set $J':=\Jac(X'), \tJ':=\Jac(X', D')$
and $M_\Tor':= \Hom(T \tJ', \Q/\Z(1))$.
We have the pull-back maps 
$f^* : J \to J', \tJ \to \tJ', M_\Tor \to M_\Tor'$ 
and the push-forward maps 
$f_* : J' \to J, \tJ' \to \tJ, M_\Tor' \to M_\Tor$.

\begin{definition}\label{def:tSg}
We define
\[ \Sigma(f):=\ker(f^* : J(\ol{k}) \to J'(\ol{k})), \quad
\Sigma_D(f):=\ker(f^* : M_\Tor \to M_\Tor').
\]
\end{definition}

Note that $\Sigma(f) \subset J[d]$ with $d:=\deg(f)$,
because $f_* \circ f^* = d$.
It follows that $\Sigma(f)$ is finite and hence
$\Sigma(f) \subset J_\Tor$.
Similarly, $\Sigma_D(f) \subset M_\Tor'[d]$ is finite too.
The canonical maps
$J_\Tor \to M_\Tor$ and $J_\Tor' \to M_\Tor'$
from \eqref{eq:MforX}
induce $\Sigma(f) \to \Sigma_D(f)$.

\begin{lemma}
\begin{enumerate}
\item 
The map $\Sigma(f) \to \Sigma_D(f)$
is injective.
\item 
Suppose that $\mathrm{gcd}\{e(x'/x) \mid x' \in f^{-1}(x)\}=1$
for all $x \in D$,
where $e(x'/x)$ denotes the ramification index.
Then $\Sigma(f) \to \Sigma_D(f)$ is an isomorphism.
\end{enumerate}
\end{lemma}
\begin{proof}
Let $L_D$ and $L_{D'}$ be the character groups
from \eqref{eq:def-LGD}.
By Lemma \ref{prop:str-mtor},
we have a commutative diagram with exact rows:
\[
\xymatrix{
0 \ar[r] &
J_\Tor \ar[r] \ar[d]^{f^*} &
M_\Tor \ar[r] \ar[d]^{f^*} &
L_D \otimes \Q/\Z \ar[r] \ar[d] &
0
\\
0 \ar[r] &
J_\Tor' \ar[r] &
M_\Tor' \ar[r] &
L_{D'} \otimes \Q/\Z \ar[r] &
0.
}
\]
This shows (1).
For (2), it suffices to show 
the injectivity of the right vertical map
under the stated assumption.
We have a commutative diagram
with exact rows:
\[
\xymatrix{
0 \ar[r] &
L_D \otimes \Q/\Z \ar[r] \ar[d] &
\underset{x \in D}{\bigoplus} \Q/\Z \ar[r]^-{\text{sum}} \ar[d] &
\Q/\Z \ar[r] \ar[d]^{\deg(f)} &
0
\\
0 \ar[r] &
L_{D'} \otimes \Q/\Z \ar[r] &
\underset{x' \in D'}{\bigoplus} \Q/\Z \ar[r]^-{\text{sum}} &
\Q/\Z \ar[r] &
0.
}
\]
Here $(x, x')$-component of 
the middle vertical map is
given by $0$ if $f(x') \not= x$
and by $e(x'/x)$ if $f(x')=x$.
The lemma follows.
\end{proof}

\begin{example}\label{ex:tSg-prime}
Let $N \in \Z_{>0}$ and
consider the canonical map $f_N : X_1(N) \to X_0(N)$,
with respect to divisors consisting of all cusps.
If $N$ is square free,
then $f_N$ is unramified at all cusps,
and hence $\Sigma(f_N) \cong \Sigma_D(f_N)$,
where $D$ is the sum of all cusps on $X_0(N)$.
\end{example}

\begin{lemma}\label{lem:tSg-mutype}
Let $f^\uab_X : X^\uab_f \to X$ 
(resp. $f^\uab_Y : Y^\uab_f \to Y$)
be the maximal unramified abelian subcovering of 
$f: X' \to X$ (resp. $f: Y' \to Y$).
Then we have 
\begin{align*}
&\Gal(f^{\uab}_X)\cong \Hom(\Sigma(f), \Q/\Z(1)),
\\
&\Gal(f^\uab_Y) \cong \Hom(\Sigma_D(f), \Q/\Z(1)).
\end{align*}
In particular, 
$\Sigma(f) \subset J_\Tor^\mu$ and 
$\Sigma_D(f) \subset M_\Tor^\mu$
(see \eqref{eq:def-mu}).
\end{lemma}
\begin{proof}
Since the first statement is a special case of second
(for $D=\emptyset$),
we only prove the latter.
For a $k$-scheme $V$,
we write $\ol{V}:=V \times_{\Spec k} \Spec \ol{k}$.
Let $\ol{f}^\uab_Y : \ol{Y}^\uab_f \to \ol{Y}$ 
be the maximal unramified
abelian subcovering of the base change
$\ol{f}: \ol{Y}' \to \ol{Y}$ of $f$.
Let $\alpha : \ol{V} \to \ol{Y}'$ 
be a Galois closure of $\ol{f}$.
We have a commutative diagram
\[
\xymatrix{
Y' \ar[d] \ar@/_3em/[dd]_f &
\ol{Y}' \ar[d] \ar[l] &
\ol{V} \ar@/^3em/[ddl]^\gamma \ar[dl]^\beta \ar[l]_\alpha &
\\
Y^\uab_f \ar[d]^{f^\uab_Y} &
\ol{Y}^\uab_f \ar[d]^{\ol{f}^\uab_Y} \ar[l] &
&
\\
Y &
\ol{Y}. \ar[l] &
}
\]
The left two squares are Cartesian 
since $Y, Y'$
(and hence $Y^\uab_f$ too) 
are geometrically integral over $k$.
It follows that
$\Gal(f_Y^\uab) \cong \Gal(\ol{f}_Y^\uab)$.
Using \eqref{eq:hc-tj2} 
and \eqref{eq:pi-H1-duality},
we get 
\begin{align*}
\Sigma_D(f)^*(1)
&=\ker(M_\Tor \to M_\Tor')^*(1)
\\
&\cong \ker(H^1(\ol{Y}, \Q/\Z(1)) \to H^1(\ol{Y}', \Q/\Z(1)))^*(1)
\\
&\cong \ker(H^1(\ol{Y}, \Q/\Z) \to H^1(\ol{Y}', \Q/\Z))^*
\\
&\cong \coker(\pi_1(\ol{Y}')^\ab \to \pi_1(\ol{Y})^\ab)
\\
&\cong\Gal(\gamma)/\Gal(\beta)
\cong\Gal(\ol{f}^\uab_Y)
\cong\Gal(f^\uab_Y).
\end{align*}
%
%
The last statement follows since
$G_k$ acts on $\Gal(f^\uab_Y)$ trivially.
\end{proof}

Note that this lemma implies
$\Sigma(f)=\Sigma(f_X^\uab)$ and
$\Sigma_D(f)=\Sigma_D(f_Y^\uab)$.

\begin{remark}\label{def:geomax}
Suppose that $f : X' \to X$
(resp. $f: Y' \to Y$) is a finite abelian unramified covering.
Then $f$ is
geometrically maximal
(see Definition \ref{def:geo-max-cov})
if and only if 
$\Sigma(f)=J_\Tor^\mu$ (resp. $\Sigma_D(f)=M_\Tor^\mu$).
\end{remark}

\begin{proposition}\label{prop:ex-gm}
Suppose that $k$ is finitely generated over its prime subfield.
Given $Y$,
there exists a geometrically maximal
abelian unramified covering $f: Y' \to Y$.
\end{proposition}
\begin{proof}
This follows from Proposition \ref{prop:exact-seq}
and Remark \ref{rem:KL}.
\end{proof}

\section{Modular curves}\label{sec:modular}
In this section, we are going to prove Theorem \ref{thm:main1}.

\subsection{First reduction}
Let $p$ be a prime
and set $k=\Q, ~X=X_0(p), ~Y=Y_0(p)$.
Recall that $D=X \setminus Y$ 
consists of two $\Q$-rational points
(i.e. $0$ and $\infty$ cusps).
It follows
that $L_D=\Z$ with trivial $G_\Q$-action,
and hence $(L_D \otimes \Q/\Z)^\mu(-1) \cong \Z/2\Z$
 (see Example \ref{ex:two}).
By \eqref{eq:exact-mtor}, we get an exact sequcne
\[ 0 \to J_\Tor^\mu(-1) \to M_\Tor^\mu(-1) \to \Z/2\Z. \]
Since $J_\Tor^\mu(-1)$ is a cyclic group of order $N_p$
by Theorem \ref{thm:mazur},
the order of $M_\Tor^\mu(-1)$ is
either $N_p$ or $2N_p$.
In view of 
Proposition \ref{prop:exact-seq} and Lemma \ref{lem:tSg-mutype},
this implies that
if 
$\ul{f}_p' : X_2'(p) \to X_0(p)$ 
is a cyclic covering of degree $2N_p$ 
such that $X_2'(p)$ is geometrically integral
and such that 
the condition (2) in Theorem \ref{thm:main1} holds,
then it automatically satisfies (1) as well.
We are going to construct such $\ul{f}_p'$.
If $p=2$ or $3$, we may take 
$\ul{f}_p'$ to be the map $\G_{m, k} \to \G_{m, k}, ~x \mapsto x^2$
under the identification 
$Y_0(p) \cong \G_{m, k} =\P^1_k \setminus \{ 0, \infty \}$.
Below we assume $p \ge 5$.
We need a preparation. 

\subsection{Generalized Dedekind eta functions}\label{sec:ded-eta}
Let $N$ be a positive integer. For an integer $g$ not congruent to $0$
modulo $N$, define the 
{\it generalized Dedekind eta function} $E_g$ by
\begin{equation}\label{eq:ded-eta}
  E_g(\tau)=q^{NB(g/N)/2}\prod_{m=1}^\infty(1-q^{N(m-1)+g})(1-q^{Nm-g}),
\end{equation}
where $B(x)=x^2-x+1/6$ is the second Bernoulli polynomial
(see \cite{Schoeneberg}). Up to scalars, these functions are also
called Siegel functions (see \cite{Kubert-Lang}).
We have the following properties of $E_g$.

\begin{Proposition}[{\cite[Corollaries 2 and 3 and Lemma 2]{Yang}}]
    \label{proposition: Eg}
\begin{enumerate}
\item We have
\begin{equation}
\label{shifting for Eg}
  E_{g+N}=E_{-g}=-E_g.
\end{equation}
\item Let $\gamma=\begin{pmatrix}a&b\\ cN&d\end{pmatrix}\in
\Gamma_0(N)$. We have, for $c=0$,
$$
  E_g(\tau+b)=e^{\pi ibNB(g/N)}E_g(\tau),
$$
and, for $c\neq 0$,
$$
  E_g(\gamma\tau)=\epsilon(a,bN,c,d)e^{\pi i(g^2ab/N-gb)}
  E_{ag}(\tau),
$$
where
\begin{equation*}
  \epsilon(a,b,c,d)
  =\begin{cases}
    e^{2\pi i\left(bd(1-c^2)+c(a+d-3)\right)/12},
      &\text{if }c\text{ is odd}, \\
    -ie^{2\pi i\left(ac(1-d^2)+d(b-c+3)\right)/12},
      &\text{if }d\text{ is odd}.
  \end{cases}
\end{equation*}
\item Suppose that $\prod_gE_g^{e_g}$ is a product of generalized
  Dedekind eta functions satisfying
$$
  \sum_ge_g\equiv 0\mod 12, \qquad \sum_gge_g\equiv 0\mod 2, \qquad
  \sum_gg^2e_g\equiv 0\mod 2N.
$$
Then $\prod_gE_g^{e_g}$ is a modular function on $\Gamma_1(N)$.
Moreover, if $N$ is odd, then the conditions can be reduced to
$$
  \sum_ge_g\equiv 0\mod 12, \qquad \sum_gg^2e_g\equiv 0\mod N.
$$
\item Given a matrix
$$
  \gamma=\begin{pmatrix}a&b\\ c&d\end{pmatrix}\in\SL(2,\Z),
$$
the Fourier expansion of $E_g(\gamma\tau)$ starts from
$\zeta q^\delta+(\text{\rm higher powers})$, where $\zeta$ is a
root of unity and
$$
  \delta=\frac{(c,N)^2}{2N}P_2\left(\frac{ag}{(c,N)}\right),
$$
where 
$P_2(x)=\{x\}^2-\{x\}+1/6$ is the second Bernoulli function,
$\{ x \}$ denotes the fraction part of $x$,
and $(a, b):=\gcd(a, b)$.
\end{enumerate}
\end{Proposition}

\subsection{The modular curve $X_2'(p)$}
Let $p$ be a prime with $p\ge 5$ and $g$ be an odd generator of $(\Z/p\Z)^\times$.
To ease the notation, set
$$
  k=N_p=\frac{p-1}{(p-1,12)}, \qquad \ell=\frac{(p-1,12)}2.
$$
Let $\Gamma_2(p)$ be the group generated by $\Gamma_1(p)$ and any
matrix of the form $\SM{g^k}\ast p\ast$ and $X_2(p)$ be the
corresponding modular curve.
By \cite[\S II, 2]{Mazur}, $X_2(p)\to X_0(p)$ is the maximal
unramified subcover of $X_1(p)\to X_0(p)$. 
We have
$$
  [\Gamma_0(p):\Gamma_2(p)]=k, \qquad [\Gamma_2(p):\Gamma_1(p)]=\ell.
$$

For an integer $h$ not congruent to $0$ modulo $p$, 
using \eqref{eq:ded-eta}, we set 
\begin{equation}\label{eq:f-h}
  F_h(\tau)=\left(\prod_{j=0}^{\ell-1}E_{g^{jk}h}(\tau)\right)^{6/\ell}.
\end{equation}

\begin{Lemma} \label{lemma: shifting for Fh}
The functions $F_h$ have the following properties.
\begin{enumerate}
\item $F_h=F_{-h}=F_{p+h}$.
\item $F_{g^kh}=\begin{cases} -F_h, &\text{if }p\equiv 1\mod 4, \\
  F_h, &\text{if }p\equiv 3\mod 4.\end{cases}
  $
\end{enumerate}
\end{Lemma}

\begin{proof} The fact that $F_h=F_{-h}$ follows immediately from
  \eqref{shifting for Eg} in Proposition \ref{proposition: Eg} since
  $F_h$ is a product of $6$ generalized Dedekind eta functions.
  By the same property of the generalized Dedekind eta functions, we
  have
  $$
    F_{p+h}=\left(\prod_{j=0}^{\ell-1}E_{g^{jk}(p+h)}\right)^{6/\ell}
  =(-1)^{6(1+g^k+\cdots+g^{(\ell-1)k})/\ell}F_h
  $$
  Since $g$ is assumed to be odd, we find that $F_{p+h}=F_h$. We next
  prove Part (2).

  By the definition of $F_h$, we have
  $$
    F_{g^kh}=\left(\frac{E_{g^{k\ell}h}}{E_h}\right)^{6/\ell}F_h.
  $$
  Now $g^{k\ell}=g^{(p-1)/2}\equiv -1\mod p$. Hence from \eqref{shifting
    for Eg} in Proposition \ref{proposition: Eg} again, we find that
  $$
    E_{g^{k\ell}h}=(-1)^{(g^{k\ell}+1)h/p}E_{-h}=-E_h
  $$
  It follows that
  $$
    F_{g^kh}=(-1)^{6/\ell}F_{h}=\begin{cases}
    -F_{h}, &\text{if }p\equiv 1\mod 4, \\
    F_h, &\text{if }p\equiv 3\mod 4.\end{cases}
  $$
  This completes the proof.
\end{proof}

Let $\psi$ be the character of order
$2$ of $\SL(2,\Z)$ defined by
$$
  \psi(\gamma)=\begin{cases}
  (-1)^{a+d-1}, &\text{if }c\text{ is odd}, \\
  (-1)^b, &\text{if }c\text{ is even}, \end{cases}
$$
for $\gamma=\SM abcd\in\SL(2,\Z)$. In addition, if $\ell$ is even,
i.e., if $p\equiv 1\mod 4$, there is a unique character $\chi$ of
order $2$ of $\Gamma_2(p)$ with
$\Gamma_1(p)\subset\operatorname{ker}\chi$. Explicitly, for
$\gamma=\SM abcd\in\Gamma_2(p)$, let $n$ be an integer such that
$a\equiv g^{nk}\mod p$. Then $\chi(\gamma)$ is equal to $(-1)^n$.
Note that if $p\equiv 5\mod 8$, then $k$ is odd and $\chi$ is simply
the character $\SM abcd\mapsto\JS dp$ of nebentype.

Let
  $\Gamma_2'(p)$ be the kernel of the character
  $$
    \begin{cases}
    \psi\chi, &\text{if }p\equiv 1\mod 4, \\
    \psi, &\text{if }p\equiv 3\mod 4
    \end{cases}
  $$
  on $\Gamma_2(p)$.
  
\begin{lemma} 
The group $\Gamma_2'(p)$ is a normal subgroup of
  $\Gamma_0(p)$ and $\Gamma_0(p)/\Gamma_2'(p)$ is cyclic of order $2k$.
\end{lemma}
\begin{proof}
Let $\rho$ be a character on $\Gamma_0(p)$ of
order $(p-1)/2$ with kernel $\Gamma_1(p)$.
By definition, $\Gamma_2'(p)$ is the
kernel of the character $\rho^\ell \psi$ (resp. $\rho^{\ell/2} \chi$) 
if $\ell$ is odd (resp. even)
on $\Gamma_0(p)$ of order $2k$.
\end{proof}

Combined with this lemma,
Theorem \ref{thm:main1} follows from the following result:

\begin{proposition}\label{prop:final}
The modular curve $X_2'(p)$ associated to $\Gamma_2'(p)$
admits a geometrically integral model over $\Q$
and the covering $\ul{f}_p' : X_2'(p)\to X_2(p)$ ramifies precisely at each cusp.
\end{proposition}

We first prove the proposition assuming $p\not\equiv 11\mod 12$. 

\begin{Lemma}
Assume that $\ell\neq 1$, i.e., that $p\not\equiv11\mod 12$. Let $h$
be an integer not congruent to $0$ modulo $p$.
\begin{enumerate}
\item For $\gamma=\SM abcd\in\Gamma_0(p)$, we have $F_h(\gamma\tau)=\psi(\gamma)F_{ah}(\tau)$.
\item Let $\gamma=\SM abcd\in\Gamma_2(p)$. Then
  $$
    F_h(\gamma\tau)=\begin{cases}
    \psi(\gamma)\chi(\gamma)F_h(\tau), &\text{if }p\equiv 1\mod 4, \\
    \psi(\gamma)F_h(\tau), &\text{if }p\equiv 3\mod 4.\end{cases}
  $$
\item For any $h$, $F_h^2$ is a modular function on $X_2(p)$ defined over $\Q$.
\item The order of $F_h^2$ at any cusp of $X_2(p)$ is odd,
and is zero elsewhere.
\end{enumerate}
\end{Lemma}

\begin{proof}
  Assume that $\ell\neq 1$, i.e., that $p\not\equiv
  11\mod 12$. Let $\gamma=\SM ab{cp}d$ be a matrix in $\Gamma_0(p)$. 
  By Part (2) of Proposition \ref{proposition: Eg}, for any
  $j=0,\ldots,\ell-1$,
  $$
    E_{g^{jk}h}(\gamma\tau)=\epsilon(a,bp,c,d)
    e^{\pi i(g^{2jk}h^2ab/p-g^{jk}hb)}E_{g^{jk}ah}(\tau).
  $$
  Notice that $\epsilon(a,bp,c,d)$ is a $12$th root of unity
  independent of $h$ and $jk$. Thus,
  $$
    F_h(\gamma\tau)=\epsilon(a,bp,c,d)^6e^{2\pi iS/2p}F_{ah}(\tau),
  $$
  where
  $$
    S=\frac{6abh^2}\ell\sum_{j=0}^{\ell-1}g^{2jk}
  $$
  We check directly from the definition of $\epsilon$ that
  $$
    \epsilon(a,bp,c,d)^6=\epsilon(a,b,cp,d)^6
   =\psi\left(\M ab{cp}d\right)=\psi(\gamma).
  $$
  Since $\ell$ is assumed to be greater than $1$, we have
  $$
    \sum_{j=0}^{\ell-1}g^{2jk}\equiv 0\mod p.
  $$
  Also, $S$ is even since either $6/\ell$ is even or
  $\sum_{j=0}^{\ell-1}g^{2jk}$ is even. We conclude that $(2p)|S$ and
  $$
    F_h(\gamma\tau)=\psi(\gamma)F_{ah}(\tau).
  $$

  We next prove Part (2).
  Assume that $\gamma=\SM abcd\in\Gamma_2(p)$. By Part (1),
  $F_h(\gamma\tau)=\psi(\gamma)F_{ah}(\tau)$. Let $n$ be a
  nonnegative integer such that $a\equiv g^{nk}\mod p$. By Lemma
  \ref{lemma: shifting for Fh}, we have
  $$
    F_{ah}=F_{g^{nk}h}=\begin{cases}
    (-1)^nF_h=\chi(\gamma)F_h, &\text{if }p\equiv 1\mod 4, \\
    F_h, &\text{if }p\equiv 3\mod 4.\end{cases}
  $$
  This yields the formula in Part (2). 
  The first statement of Part (3) is an immediate consequence of Part (2),
  and the second statement follows immediately from the definition \eqref{eq:f-h}.

  We now prove Part (4).
  The statement for non-cusp points are obvious from 
  the definition \eqref{eq:f-h}.
  Let $a/c$, $(a,c)=1$, be a cusp of $X_2(p)$. Consider first the case
  $p\nmid c$. Let $\gamma=\SM abcd$ be a matrix in $\SL(2,\Z)$. By
  Part (4) of Proposition \ref{proposition: Eg}, the Fourier expansion
  of $E_{g^{jk}h}$ starts from the term $q^{1/12p}$ for any $j$. Since
  such a cusp has width $p$ and $F_h^2$ is the product of exactly $12$
  $E_{g^{jk}h}$, the order of $F_h$ at $a/c$ is $1$.

  Now consider the case $p|c$. Such a cusp has width $1$. By Part (4)
  of Proposition \ref{proposition: Eg}, the order of $F_h$ at $a/c$ is
  $$
    \frac{12}\ell\sum_{j=0}^{\ell-1}\frac p2\left(
    \left\{\frac{ag^{jk}h}p\right\}^2-\left\{\frac{ag^{jk}k}p\right\}
    +\frac16\right).
  $$
  Observe that for any integer $x$, if we let
  $n=\lfloor x/p\rfloor$, then
  $$
    \left\{\frac xp\right\}^2-\left\{\frac xp\right\}
   =\left(\frac xp-n\right)^2-\frac xp+n
   =\frac{x^2}{p^2}-\frac xp-2n\frac xp+n^2+n,
  $$
  and hence
  $$
   \frac p2\left(\left\{\frac xp\right\}^2-\left\{\frac
       xp\right\}\right)
  \equiv \frac p2\left(\frac{x^2}{p^2}-\frac xp\right)\mod 1.
  $$
  It follows that
  $$
    \frac{12}\ell\sum_{j=0}^{\ell-1}\frac p2\left(
    \left\{\frac{ag^{jk}h}p\right\}^2-\left\{\frac{ag^{jk}k}p\right\}
    +\frac16\right)\equiv \frac{12}\ell\sum_{j=0}^{\ell-1}
    \frac p{12}=p\equiv 1 \mod 2
  $$
  That is, the order of $F_h^2$ at the cusp $a/c$ is odd. This
  completes the proof.
\end{proof}

It follows from the lemma that
$\Q(X_2'(p))$ is a quadratic extension of $\Q(X_2(p))$
generated by a square root of $F_h^2 \in \Q(X_2(p))$,
that $\Q$ is algebraically closed in $\Q(X_2'(p))$,
and that the covering $X_2'(p) \to X_2(p)$
ramifies exactly at each cusp,
showing Proposition \ref{prop:final} in this case.

To show the proposition for the case $p\equiv 11\mod 12$,
we need 
a slightly different construction of modular functions.
In this case, $\Gamma_2(p)=\Gamma_1(p)$. 
Let $h=(h_1,h_2,h_3)$ be any triplet
of integers not congruent to $0$ modulo $p$ such that
$h_1^2+h_2^2+h_3^2\equiv 0\mod p$ and define
$G_h=(E_{h_1}E_{h_2}E_{h_3})^2$. Then $G_h^2$ is a modular function on
$X_1(p)$ by Proposition \ref{proposition: Eg}. Also, by the same
computation as above, the order of $G_h^2$ at each cusp is odd. In
addition, we can verify as above that for $\gamma=\SM
abcd\in\Gamma_1(p)$,
$$
  G_h(\gamma\tau)=\psi(\gamma)G_h.
$$
Therefore, the proposition holds for the case $p\equiv 11\mod 12$.
\qed

\begin{Remark} If $p\not\equiv 1\mod 8$, then
$$
  z(\tau)=\left(\frac{\eta(\tau)}{\eta(p\tau)}\right)^{12/(p-1,12)}
$$
is also a modular function on $\Gamma_2'(p)$, but not on
$\Gamma_2(p)$. Thus, the function field of $X_2'(p)$ can be obtained
by adjoining either $F_h(\tau)$ or $z(\tau)$. The relation betwen
$F_h$ and $z$ is
$$
  z(\tau)=\pm\prod_{j=0}^{k-1}F_{g^j}.
$$
(In particular, if $k$ is odd, i.e., if $p\not\equiv 1\mod 8$, then
$z(\gamma\tau)=\psi(\gamma)z(\tau)$ for $\gamma\in\Gamma_2(p)$.)
\end{Remark}

\vspace{3mm}
\noindent
{\it Acknowledgement.} 
The first author would like to thank 
Masataka Chida and Fu-Tsun Wei for fruitful discussion.
He is partially supported by JSPS KAKENHI Grant (18K03232).
The second author was partially supported by Grant
106-2115-M-002-009-MY3 of the Ministry of Science and Technology,
Republic of China (Taiwan).

The authors would like to thank the anonymous referee 
for the detailed comments.

\end{document}